\documentclass{article}
\usepackage{graphicx}
\usepackage[left=3cm, right=3cm]{geometry}
\usepackage[latin1]{inputenc}

\usepackage{color}
\usepackage{amsthm,amsmath}
\usepackage{amssymb}
\usepackage{amsfonts}

\newtheorem{theorem}{Theorem}
\newtheorem{dfn}{Definition}
\newtheorem{coro}{Corollary}

\newtheorem{lemma}{Lemma}
\newtheorem{example}{Example}

\newcommand{\te}{\theta}

\author{Fatmanur Gursoy, Elif Segah Oztas, Irfan Siap\\
	\small{Department of Mathematics,}\\
	\small{ Yildiz Technical University, Turkey}\\  \small{(fatmanur@yildiz.edu.tr, esoztas@yildiz.edu.tr, irfan.siap@gmail.com)}}

\begin{document}
\title{Reversible DNA Codes Using Skew Polynomial Rings\footnote{A part of this study is presented in The International Conference on Coding theory and Cryptography (ICCC2015, Algeria).}}
\maketitle
\begin{abstract}

In this study we determine the  structure of reversible DNA codes obtained from skew cyclic codes. We show that the generators of such DNA codes enjoy some special properties. We study the structural properties of such family of codes and we also illustrate our results with examples.

\end{abstract}
\section{Introduction}
 DNA is a molecule that carries most of the genetic instructions for the functions of cells. DNA sequences consist of Adenine (A), Guanine (G), Cytosine (C), Thymine (T) nucleotides. The bases (nucleotides) govern  DNA double strings with property of Watson-Crick complement (WCC).  According to WCC, A and G  are complement of T and C, respectively. This is shown as  $\text{A}^\text{c}=$T, $\text{T}^\text{c}=$A, $\text{C}^\text{c}=$G and $\text{G}^\text{c}=$C.

 The interest on the structure of DNA in computations was introduced by Adleman \cite{adleman94} who solved a famous NP-hard problem by using DNA molecules. Here, WCC property of DNA was used to solve the problem.
There are many of studies about DNA codes and DNA computing. Since we focus specifically on the skew cyclic codes and DNA codes, we mention some recent and related papers about  DNA codes.
 In \cite{tehergf4}, DNA codes are generated by additive cyclic codes over $F_4$. In \cite{tahersiapsim} and \cite{siaptaherr} authors used the ring $F_2[u]/(u^2-1)$ and obtained cyclic reversible DNA codes. In \cite{yildizsiap}, DNA double bases are considered by using the ring $F_2[u]/(u^4-1)$  with 16 elements but authors restricted the length of DNA codes with odd integers. In \cite{ise} and  \cite{Gse} (generalized form of \cite{ise}), DNA double bases are used with $F_{16}$ and DNA $2s$-bases by special $4^s$-power tables. In \cite{agul} DNA codes are generated over $F_4+vF_4$ ($v^2=v$).

In this study, skew cyclic codes are used to obtain reversible DNA codes and solve the reversibility problem.  To explain the reversibility  problem, we assume that  $(a_1,a_2,a_3)$  is a codeword corresponding to the DNA string ACAGTC where $a_1 \rightarrow $AC, $a_2 \rightarrow $AG, $a_3 \rightarrow $TC. Then, the reverse of   $(a_1,a_2,a_3)$ is $(a_3,a_2,a_1)$  and it corresponds to TCAGAC. But TCAGAC is not the reverse of ACAGTC. We solve this problem by using $\theta$-palindromic polynomial concept which is introduced in this work  and  properties of  divisors of the polynomial $x^n-1$ in the skew polynomial ring  $F_{4^{2s}}[x;\theta]$ for the first time. In previous studies, DNA codes were considered over commutative rings.

\section{Preliminaries and Definitions}

In this section we present some basic definitions and some properties  of skew cyclic codes and DNA codes.
\begin{dfn}
	Let $C$ be a code of length $n$ over $F_q$. If $c^r=(c_{n-1},c_{n-2},\ldots,c_1,c_0)\in C$ for all $c=(c_0,c_1,\ldots,c_{n-1})\in C$, then $C$ is called a reversible code.
\end{dfn}

Let $\theta$ be an automorphism over $F_q$. Then, the set of polynomials $F_q[x;\theta]=\{a_0+a_1x+...+a_{n-1}x^{n-1}|a_i\in F_q, n\in \mathbb{N}\}$ is the skew polynomial ring over $F_q$ where addition is the usual addition of polynomials and the multiplication is defined by $xa=\theta(a)x$ ($a\in F_q$)\cite{Mc}.
A skew cyclic code is defined to be a linear code $C$ of length $n$ over $F_q$ which  satisfies the property that $(\te(c_{n-1}),\te(c_0),...,\te(c_{n-2}))\in C$, for all $(c_0,c_1,...,c_{n-1})\in C$ \cite{skew cyclic}.
\begin{align*}
	\varphi: F_{q}^n & \rightarrow F_q[x;\theta]/(x^n-1) \\
	c=(c_0,c_1,\ldots,c_{n-1}) &\rightarrow c(x)=c_0+c_1x+\ldots+c_{n-1}x^{n-1}
\end{align*}
In polynomial representation of a code of length $n$, we associate a codeword $c$ with the element $\varphi(c)=c(x)$ of $F_q[x;\theta]/(x^n-1)$. In this way a skew cyclic code $C$ of length $n$ over $F_q$ corresponds to a left ideal of the quotient ring $F_q[x;\theta]/(x^n-1)$, if the order of $\theta$, say $m$, divides $n$ \cite{skew cyclic}. If $m$ does not divide $n$, then $F_q[x;\theta]/(x^n-1)$ is not a ring anymore. In this case the skew cyclic code $C$ can be  considered as a left $F_q[x;\te]$-submodule of  $F_q[x;\te]/(x^n-1)$ \cite{I.siap}. In both cases $C$ is generated by a monic polynomial $g(x)$ which is a right divisor of $x^n-1$ in $F_q[x;\te]$ and denoted by $C=\langle g(x)\rangle$. Moreover if $m$ and $n$ are relatively prime, then $C$ is a cyclic code over $F_q$ \cite{I.siap}.

Naturally skew cyclic codes are linear codes over $F_q$. A linear code of length $n$ with dimension $k$ and minimum Hamming distance $d$ is denoted by $[n,k,d]$ code.

In \cite{ise} authors introduced $4$-lifted polynomials over $F_{16}$ to obtain reversible DNA codes. They used Table  \ref{dna table} which maps each element of $F_{16}$ and its $4$th power to the DNA pairs which are reverses of each other. For example
 $\alpha^2$ is mapped to the DNA pair GC and    $\alpha^8$ is mapped to the DNA pair CG.

 After that in \cite{Gse} they gave an algorithm  building $4^s$-table for the  correspondence  between $F_{4^{2s}}$ and the DNA 2s-bases and by using the $4^s$-lifted  polynomials they obtained reversible DNA codes.
According to this algorithm we define a map; 
\begin{align*}
\tau:F_{4^{2s}} & \rightarrow \{\text{A,T,G,C}\}^{2s} \\
\beta & \rightarrow (b_0,b_1,...,b_{2s-1}).
\end{align*}
This map can be naturally extended to a map $\phi$ from $F_{4^{2s}}^n$ to $\{\text{A,T,G,C}\}^{2sn}$ as follows; $\phi(c_0,c_1,...,c_{n-1})
= (\tau(c_0),\tau(c_1),\ldots, \tau(c_{n-1}))$ where $c_i\in F_{4^{2s}}$, $i\in\{0, \ldots, n-1\}$.
For instance, $\phi(c_0,c_1,c_2,c_{3})=\phi(\alpha^2,\alpha^{20},$ $  \alpha^{135}, \alpha^{219})=
(\text{AAAGATACCGACTAGA})$ in $F_{256}$ where $\alpha$ is a primitive element of $F_{256}$ (by using the correspondence table in \cite{Gse}).
\begin{dfn}
Let $C\subseteq F_{4^{2s}}^n$. If $\phi(c)^r\in \phi(C)$  for all $c\in C$, then $C$ or equivalently $\phi(C)$ is called a reversible DNA code.
\end{dfn}
According to the  $4^s$-power correspondence in \cite{Gse} we observe that $\tau(\beta)$ and $\tau(\beta^{4^s})$ are reverses of each other for all $\beta \in F_{4^{2s}} $. Thus the reverse of $\phi(c_0,c_1,\ldots,c_{n-1})$ is  $\phi(c_{n-1}^{4^s},\ldots,c_1^{4^s},c_0^{4^s})$. For convenience, we will say that $(c_0,c_1,\ldots,c_{n-1})$ and  $(c_{n-1}^{4^s},\ldots,c_1^{4^s},c_0^{4^s})$ are DNA  reverses of each other. The skew polynomial ring $F_{4^{2s}}[x;\te]$ where $\te(a)=a^{4^s}$ for all $a\in F_{4^{2s}}$, studied in this paper, resolves this approach naturally and proves to be more effective as it will be seen later.
In both \cite{ise} and \cite{Gse} the dual codes of reversible DNA codes are not considered however, here we also study their duals too.

In this study we construct reversible DNA codes by using the skew
cyclic codes over $F_{4^{2s}}$. We impose some additional properties to the generator polynomials of skew cyclic codes over  $F_{4^{2s}}$ and using the correspondence algorithm given in \cite{Gse} we obtained reversible DNA codes directly.
We also showed that the dual codes of these skew cyclic codes are also reversible DNA codes.

\begin{table}
	
	\begin{center}
		
		\caption{ \cite{ise}} \label{dna table}
		\begin{tabular}{lcl}
			\hline
			Double DNA pair& $F_{16}$(multiplicative)&  additive \\
			\hline
			AA & 0 & 0 \\
			TT & $\alpha^0$ &1 \\
			AT & $\alpha^1$ &$\alpha$\\
			GC & $\alpha^2$ & $\alpha^2$\\
			AG & $\alpha^3$ & $\alpha^3$\\
			TA & $\alpha^4$ & $1+\alpha$ \\
			CC & $\alpha^5$ & $ \alpha+\alpha^2$\\
			AC & $\alpha^6$ & $\alpha^2 +\alpha^3$\\
			GT & $\alpha^7$ & $1+\alpha +\alpha^3$\\
			CG & $\alpha^8$ & $1 +\alpha^2$\\
			CA & $\alpha^9$ & $\alpha +\alpha^3$\\
			GG & $\alpha^{10}$ & $1+\alpha +\alpha^2$\\
			CT & $\alpha^{11}$ & $\alpha +\alpha^2+\alpha^3$\\
			GA & $\alpha^{12}$ & $1+\alpha +\alpha^2+\alpha^3$\\
			TG & $\alpha^{13}$ & $1 +\alpha^2+\alpha^3$\\
			TC & $\alpha^{14}$ & $1+\alpha^3$\\
		\end{tabular}
	\end{center}
\end{table}

\begin{dfn}
	Let $f(x)=a_0+a_1x+\ldots +a_tx^t$ be a polynomial of degree $t$ over  $F_q$ and $\te$ be an automorphism of $F_q$.  $f(x)$ is said to be a palindromic polynomial if $a_i=a_{t-i}$  for all $i\in  \{0,1,\ldots,t \}$,
	and $f(x)$ is said to be a $\theta$-palindromic polynomial if $a_i=\theta(a_{t-i})$  for all $i\in \{0,1,\ldots,t\}$.
\end{dfn}

\begin{dfn}\cite{note on dual}
	The skew reciprocal polynomial of $f(x)=\sum_{i=0}^{t}a_ix^i\in F_q[x,\theta]$ of degree $t$ is defined as
	$$f^R(x)=\sum_{i=0}^{t}x^i a_{t-i}=\sum_{i=0}^{t}\theta^i(a_{t-i})x^{i}.$$
	If $f(x)=f^R(x)$, then $f(x)$ is called  a skew self reciprocal polynomial.
\end{dfn}
\begin{lemma}\cite{coding with}
	Suppose that the order of $\theta$ divides $n$. Let $x^n-1=h(x)g(x)$ in $F_q[x,\theta]$ and $C$ be the skew cyclic code of length $n$ over $F_q$ generated by $g(x)$. Then, the dual of $C$ is a skew cyclic code of length $n$ generated by $h^R(x)$, i.e. $C^\perp=\langle h^R(x)\rangle$.
\end{lemma}
We emphasize that  skew self reciprocal polynomials are different than  the  $\theta$-palindromic polynomials. We illustrate this fact with the following example.
\begin{example}
	Let $\alpha$ be a primitive element of $F_{16}$ and $f(x)=1+\alpha x+\alpha^2 x^2+\alpha^4 x^3+x^4\in F_{16}[x;\te]$ where $\te(a)=a^4$. Then, $f(x)=f^R(x)$ but since $\theta(\alpha^2)\neq\alpha^2$, $f(x)$ is not a $\theta$-palindromic polynomial.
\end{example}
\section{ Reversible DNA codes }

In this study we use the skew polynomial ring $F_{q}[x;\theta]$ with the automorphism $\theta$ on $F_{q}$ defined by $\theta(a)=a^{4^s}$ where $q=4^{2s}$. Since the order of $\theta$ is 2 we note that skew cyclic codes of odd length over $F_{q}$ with respect to $\theta$ are ordinary cyclic codes (\cite{I.siap},Theorem 8). For this reason we investigate the DNA reversible codes of even length and  odd length separately.
\subsection{Even length case}\label{32}

In this subsection we  deal with the divisors of $x^n-1$ where $n$ is even. Let $C$ be a skew cyclic code of length $n$ over $F_{q}$ with respect to the automorphism $\theta$ and $g(x)$ be the monic nonzero polynomial of minimal degree in $C$. Then, $C$ is generated by $g(x)$ and moreover $g(x)$ is a right divisor of $x^n-1$ in $F_{q}[x;\theta]$ (Lemma 1 in \cite{skew cyclic}). It is easily seen that the ideals $\langle g(x)\rangle$ and $\langle \beta g(x)\rangle$ where $\beta \in F^*_{q}$ are equal in $F_{q}[x;\theta]$ and if
$x^n-1=h(x)g(x)$, then $x^n-1=(h(x)\beta^{-1})(\beta g(x))$ in $F_{q}[x;\theta]$. Hence if $g(x)$ is a right divisor of $x^n-1$, then so is $\beta g(x)$. Therefore the generating polynomial need not  be monic. For the following theorems, we take $g(x)=g_0+g_1x+\ldots +g_mx^m$ as the generating polynomial of $C$,  which is a polynomial of minimal degree in $C$ without considering its monicness.
\begin{theorem}\label{teo}
Let $g(x)$ be a right divisor of $x^n-1$ in $F_{q}[x;\theta]$ where $deg(g(x))=m$ is even.
 Then, the skew cyclic code $C=\langle g(x)\rangle$ over $F_{q}$ with length $n$ is a reversible DNA code if and only if  $g(x)$ is a palindromic polynomial.
\end{theorem}
\begin{proof}
	 Let $g(x)$ be a palindromic polynomial. Recall that $\phi$ gives the correspondence of codewords in DNA form. Reverses of each DNA codeword $\phi(c)$, for $c\in C$, are obtained by the following equation:
	 \begin{equation}
	 \left( \phi \left(  \sum_{i=0} ^{k-1}\beta_i x^i g(x)\right)\right)^r =  \phi  \left(  \sum_{i=0} ^{k-1} \theta(\beta_i) x^{k-1-i}g(x) \right)
	 \end{equation}
	 where $k=n-deg(g(x))$ and $\beta_i \in F_{q}$. Since $ \sum_i \theta(\beta_i) x^{k-1-i}g(x)\in C$, then $C$ is a  reversible DNA code.\\
	 Conversely, let $C=\langle g(x)\rangle$ be a reversible DNA code and $a(x)=g_m^{-1}g(x)=a_0+a_1x+\ldots +a_{m-1}x^{m-1}+x^m$ where $a_i=g_m^{-1}g_i$. Then, $a(x)\in C$ is the nonzero monic polynomial of minimal degree in C. Since
	 $c_1(x)=x^{n-m-1}a(x)=a_0^{4^s}x^{n-m-1}+a_1^{4^s}x^{n-m}+\ldots+a_{m-1}^{4^s}x^{n-2}+x^{n-1}\in C$, then its DNA reverse
	 $c_2(x)=1+a_{m-1}x+\ldots+a_1x^{m-1}+a_0x^m$  is in $C$. So
	 $c_3(x)=a(x)-a_0^{-1}c_2(x)=(a_0-a_0^{-1})+(a_1-a_0^{-1}a_{m-1})x+\ldots+(a_{m-1}-a_0^{-1}a_1)x^{m-1}$ is in $C$ with degree less than $deg(g(x))=m$, which contradicts with the minimality of $deg(g(x))$ if  $c_3(x)$ is nonzero. Hence
	
	 $$ c_3(x)=0\Rightarrow a_0-a_0^{-1}=0\Rightarrow a_0=1,$$
	 also $a_1-a_0^{-1}a_{m-1}=0\Rightarrow a_1-a_{m-1}=0\Rightarrow a_1=a_{m-1} $. Continuing in this manner we obtain that $a_i=a_{m-i}$ for all $i\in \{0,1,\ldots,m\}$. So $a(x)$ is a palindromic polynomial. Then,  $g(x)=g_m a(x)$ is also a palindromic polynomial.
\end{proof}

\begin{theorem}\label{teo1}
Let $g(x)$ be a right divisor of $x^n-1$ in $F_{q}[x;\theta]$ where $deg(g(x))=m$ is odd and $C=\langle g(x)\rangle$ be a skew cyclic code of length $n$ over $F_{q}$.
If $g(x)$ is a $\theta$-palindromic polynomial, then $C$ is a reversible DNA code.
Conversely if $C$ is a reversible DNA code, then $C$ is generated by a $\theta$-palindromic polynomial.
\end{theorem}
\begin{proof}
	  Let $g(x)$ be a $\theta$-palindromic polynomial. Recall that $\phi$ gives the correspondence of codewords in DNA form. Reverses of each DNA codeword $\phi(c)$, for $c\in C$, are obtained by the following equation:
	  \begin{equation}
	  \left( \phi \left(  \sum_{i=0}^{k-1} \beta_i x^i g(x)\right)\right)^r =  \phi  \left(  \sum_{i=0}^{k-1} \theta(\beta_i) x^{k-1-i}g(x) \right)
	  \end{equation}
	  where $\beta_i \in F_{q}$ and $k=n-deg(g(x))$. Since $ \sum_{i=0}^{k-1} \theta(\beta_i) x^{k-1-i}g(x)\in C$, then $C$ is a  reversible DNA code.
	
	  Let $C=\langle g(x)\rangle$ be a reversible DNA code and $a(x)=g_m^{-1}g(x)=a_0+a_1x+\ldots +a_{m-1}x^{m-1}+x^m$ where $a_i=g_m^{-1}g_i$. Then, $a(x)\in C$ is the nonzero monic polynomial of minimal degree in C. Since
	  $c_1(x)=x^{n-m-1}a(x)=a_0x^{n-m-1}+a_1x^{n-m}+\ldots+a_{m-1}x^{n-2}+x^{n-1}\in C$,
	  then its DNA reverse
	  $c_2(x)=1+a_{m-1}^{4^s}x+\ldots+a_1^{4^s}x^{m-1}+a_0^{4^s}x^{m}\in C$. So, $c_3(x)=a(x)-a_0^{-4^s}c_2(x)=(a_0-a_0^{-4^s})+(a_1-a_0^{-4^s}a_{m-1}^{4^s})x+\ldots+(a_{m-1}-a_0^{-4^s}a_1^{4^s})x^{m-1}$
	  is in $C$ with degree less than $deg(g(x))=m$, this contradicts with the minimality of $deg(g(x))$ if  $c_3(x)$ is nonzero. Hence
	  $$ c_3(x)=0\Rightarrow a_0-a_0^{-4^s}=0\Rightarrow a_0^{4^s+1}=1\Rightarrow a_0=\alpha^{(4^s-1)j} $$
	$ \text{where } \alpha \text{ is a primitive element of  } F_q$, and $j$ is a positive integer. Also $a_{m-1}-a_0^{-4^s}a_1^{4^s}=0\Rightarrow a_{m-1}=\alpha^{(4^s-1)j}a_1^{4^s}.$ Continuing in this
	  manner we obtain that $$a(x)=\sum\limits_{i=0}^{\frac{m-1}{2}} (a_ix^i+\alpha^{(4^s-1)j}a_i^{4^s}x^{m-i})\text{ where } a_0=\alpha^{(4^s-1)j}.$$
	  Multiply $a(x)$ with $\alpha^j$ then,
	  $$\alpha^ja(x)=\sum\limits_{i=0}^{\frac{m-1}{2}} \alpha^ja_ix^i+\alpha^{(4^s)j}a_i^{4^s}x^{m-i}\in C.$$
	  Hence  $\alpha^ja(x)=\alpha^jg_m^{-1}g(x)$ is a $\theta$-palindromic polynomial in $C$. Since the ideal generated by $\alpha^ja(x)$ is equal to the ideal $\langle g(x)\rangle=C$, we conclude that $C$ can be generated by a $\theta$-palindromic polynomial.	
\end{proof}

Since in this subsection we consider the even case for $n$,  $x^n-1$ is in the center of the ring $F_{q}[x;\theta]$. As a consequence of Lemma 7 in \cite{coding with} we have the following lemma.
\begin{lemma} \label{lm7}
Let $x^n-1=h(x)g(x)$ in $F_{q}[x;\theta]$. Then, $x^n-1=g(x)h(x)$, i.e. any right divisor of $x^n-1$ is also a left divisor in $F_{q}[x;\theta]$.
\end{lemma}
\begin{theorem}\label{lem1}
Let $x^n-1=h(x)g(x)$ in $F_{q}[x;\theta]$ where the degree of $g(x)$ is even. If $h(x)$ is a palindromic polynomial, then $g(x)$ is also a palindromic polynomial.
\end{theorem}
\begin{proof}
	Let $h(x)=h_0+h_1x+\ldots+h_{2k}x^{2k}$ and $g(x)=g_0+g_1x+\ldots+g_{2r}x^{2r}$  where $n=2r+2k$. Suppose that $h(x)$ is a palindromic polynomial then
	$h_i=h_{2k-i}$ for all $i=0,1,\ldots,k$.
	Let $a_i$ be the coefficient of $x^i$ in $h(x)g(x)$.
	For any $t<n/2$, the coefficient of $x^t$ in $h(x)g(x)$ is $a_t=\sum_{j=0}^{t}h_j\theta^j(g_{t-j})$ and the coefficient of $x^{n-t}$ is $a_{n-t}=\sum_{j=0}^{t}h_{2k-j}\theta^{2k-j}(g_{2r-(t-j)})$.
	
	$h(x)g(x)=x^n-1$ implies that $a_0=a_n=1$ and $a_i=0$ for all $i\in {1,\ldots,n-1}$.
	We will show that $g_i=g_{2r-i}$ for all $i=0,1,\ldots,r$ by induction.
	
	For $i=0$;
	$a_0=h_0\theta^0(g_0)=h_0g_0$, on the other hand $a_n=h_{2k}\theta^{2k}(g_{2r})=h_{2k}g_{2r}$.
	Since $a_0=a_n=1$ and $h_0=h_{2k}$, then we have $g_0=g_{2r}$.
	
	Suppose the induction hypothesis $g_i=g_{2r-i}$ is true for all $0<i<l$ (where $l<r$).
	Now let us look at the coefficients $a_l$ and $a_{n-l}$;

	\begin{equation}
	a_l=\sum_{j=0}^{l}h_j\theta^j(g_{l-j})=\sum_{j=1}^{l}h_j\theta^j(g_{l-j})+h_0g_l,
	\end{equation}
	
	\begin{equation}
	a_{n-l}=\sum_{j=0}^{l}h_{2k-j}\theta^{2k-j}(g_{2r-(l-j)})=\sum_{j=1}^{l}h_{2k-j}\theta^{2k-j}(g_{2r-(l-j)})+h_{2k}g_{2r-l}.
	\end{equation}
	
	Since the order of $\te$ is $2$, then $\theta^{j}(a)=\theta^{2k-j}(a)$ for all $a\in F_{q}$ and  $j\in \{1,\ldots,l\}$.
		Since $h_j=h_{2k-j}$ and $g_{l-j}=g_{2r-(l-j)}$, then we have $h_j\theta^j(g_{l-j})=h_{2k-j}\theta^{2k-j}(g_{2r-(l-j)})$.

	Therefore $\sum_{j=1}^{l}h_j\theta^j(g_{l-j})=\sum_{j=1}^{l}h_{2k-j}\theta^{2k-j}(g_{2r-(l-j)})$.
	Since $a_l=a_{n-l}=0$ we can conclude that $h_0g_l=h_{2k}g_{2r-l}=h_0g_{2r-l}$. Thus $g_l=g_{2r-l}$.	
	Therefore $g(x)$ is a palindromic polynomial.
\end{proof}

As a consequence of Lemma \ref{lm7} and Theorem \ref{lem1} we have the following corollary:
\begin{coro}\label{coro1}
Let $x^n-1=h(x)g(x)$ in $F_{q}[x;\theta]$ where the degree of $g(x)$ is even. If $g(x)$ is a palindromic polynomial, then $h(x)$ is a palindromic polynomial.
\end{coro}

\begin{theorem}\label{teo4}
Let $x^n-1=h(x)g(x)$ in $F_{q}[x;\theta]$ where the degree of $g(x)$ is odd. If $g(x)$ is a $\theta$-palindromic polynomial, then  $h(x)$ is a palindromic polynomial.
\end{theorem}

\begin{proof}
	Let $h(x)=h_0+h_1x+\ldots+h_{2k-1}x^{2k-1}$ and $g(x)=g_0+g_1x+\ldots+g_{2r-1}x^{2r-1}$  where $n=2r+2k-2$. Suppose that $g(x)$ is a $\theta$-palindromic polynomial then
	$g_i=\te(g_{2r-1-i})$ for all $i=0,1,\ldots,r-1$.
	Let $a_i$ be the coefficient of $x^i$ in $h(x)g(x)$.
	For any $t<n/2$, the coefficient of $x^t$ in $h(x)g(x)$ is $a_t=\sum_{j=0}^{t}h_j\theta^j(g_{t-j})$ and the coefficient of $x^{n-t}$ is $a_{n-t}=\sum_{j=0}^{t}h_{2k-1-j}\theta^{2k-1-j}(g_{2r-1-(t-j)})$.
	
	$h(x)g(x)=x^n-1$ implies that $a_0=a_n=1$ and $a_i=0$ for all $i\in {1,\ldots,n-1}$.
	We will show that $h_i=h_{2k-1-i}$ for all $i=0,1,\ldots,k-1$ by induction.
	
	For $i=0$;
	$a_0=h_0\theta^0(g_0)=h_0g_0$, on the other hand $a_n=h_{2k-1}\theta^{2k-1}(g_{2r-1})=h_{2k-1}\theta(g_{2r-1})$.
	Since $a_0=a_n=1$ and $g_0=\theta(g_{2r-1})$, then we have $h_0=h_{2k-1}$.
	
Suppose the induction hypothesis $h_i=h_{2k-1-i}$ is true for all $0<i<l$ (where $l\leq k-1$).
	Now let us look at the coefficients $a_l$ and $a_{n-l}$;
	\begin{align*}	
		a_l&=\sum_{j=0}^{l}h_j\theta^j(g_{l-j})=\sum_{j=0}^{l-1}h_j\theta^j(g_{l-j})+h_l\theta^l(g_0),\\
	a_{n-l}&=\sum_{j=0}^{l}h_{2k-1-j}\theta^{2k-1-j}(g_{2r-1-(l-j)})\\
	&=\sum_{j=0}^{l-1}h_{2k-1-j}\theta^{2k-1-j}(g_{2r-1-(l-j)})+h_{2k-1-l}\theta^{2k-1-l}(g_{2r-1}).
	\end{align*}

Since the order of $\te$ is $2$, then $\theta^{j}(a)=\theta^{2k-1-j}(\theta(a))$ for all $a\in F_{q}$ and  $j\in \{0,\ldots,l-1\}$.	
		Since $g_{l-j}=\te(g_{2r-(l-j)})$  and $h_j=h_{2k-1-j}$, then we have
		 $\theta^{j}(g_{l-j})=\theta^{2k-1-j}(\theta(g_{l-j}))=\theta^{2k-1-j}(g_{2r-1-(l-j)})$ and hence
		$h_j\theta^j(g_{l-j})=h_{2k-1-j}\theta^{2k-1-j}(g_{2r-1-(l-j)})$.\\
	Therefore $\sum_{j=0}^{l-1}h_j\theta^j(g_{l-j})=\sum_{j=0}^{l-1}h_{2k-1-j}\theta^{2k-1-j}(g_{2r-1-(l-j)})$.
	Since $a_l=a_{n-l}=0$, we can conclude that $h_lg_0=h_{2k-1-l}\theta(g_{2r-1})$. Thus $h_l=h_{2k-1-l}$.	
	Therefore $h(x)$ is a palindromic  polynomial.
\end{proof}


\begin{theorem}\label{teo5}

Let $h(x)$ be a palindromic polynomial in $F_{q}[x;\theta]$ of degree $t$.
\begin{enumerate}
  \item If $t$ is an  odd integer,  then $h^R(x)$ is a $\theta$-palindromic polynomial.
  \item If $t$ is an  even integer, then $h^R(x)$ is also a palindromic polynomial.
\end{enumerate}
\end{theorem}

\begin{proof}
	 Let $h(x)=h_0+h_1x+\ldots+h_{t-1}x^{t-1}\in F_{q}[x;\theta]$ be a palindromic polynomial. Then, $h_i=h_{t-i}$ for all $i=0,1,\ldots,t-1$ and $h^R(x)=\sum_{i=0}^{t}a_i x^i=\sum_{i=0}^{t}\theta^i(h_{t-i})x^{i}$. 
	 \begin{enumerate}
	 	
	 	\item Suppose that $t$ is   odd.
	 	If $i$ is an odd number, then $t-i$ is even.
	 	$$a_i=\theta^i(h_{t-i})=\theta(h_{t-i})=\theta(h_{i})\mbox{  and  } a_{t-i}=\theta^{t-i}(h_{i})=h_i.$$
	 	Hence $a_i=\theta(a_{t-i})$. Similarly $a_i=\theta(a_{t-i})$ where $i$ is even.
	 	Thus we can conclude that $h^R(x)$ is a $\theta$-palindromic polynomial.
	 	\item Suppose that $t$ is even.
	 	If $i$ is an odd number, then $t-i$ is also an odd number.
	 	$$a_i=\theta^i(h_{t-i})=\theta(h_{t-i})=\theta(h_{i})\mbox{  and  } a_{t-i}=\theta^{t-i}(h_{i})=\theta(h_i).$$
	 	Thus $a_i=a_{t-i}$. Similarly $a_i=a_{t-i}$ where $i$ is even.
	 	Hence we obtain that $h^R(x)$ is  a palindromic polynomial.
	 \end{enumerate}	
\end{proof}

\begin{coro}\label{lem4}

Let $C$ be a skew cyclic code  of length $n$ over $F_{q}$ with respect to the automorphism $\te$. If $C$ is a reversible DNA code, then so is $C^\perp$.

\end{coro}

\begin{proof}
	Let $C$ be a skew cyclic code and $g(x)$ be a nonzero polynomial of minimal degree in $C$. Then $C$ is generated by $g(x)$ and  $g(x)$ is a right divisor of $x^n-1$ in $F_{q}[x;\theta]$ (\cite{skew cyclic}). 
	Suppose that $C$ is a reversible DNA code. Then we have two cases:
	\begin{enumerate}
		\item If $deg(g(x))$ is an  even integer,  then by Theorem \ref{teo} we have $g(x)$ is a palindromic polynomial. If $x^n-1=h(x)g(x)$ in $F_{q}[x;\theta]$, then as a result of Corollary \ref{coro1} and Theorem \ref{teo5}; $h^R(x)$ is a palindromic polynomial. Hence $C^\perp= \langle h^R(x) \rangle$ is a reversible DNA code, by Theorem \ref{teo}.

		\item If $deg(g(x))$ is an  odd integer,  then by Theorem \ref{teo1}; $C$ can be generated by a   $\theta$-palindromic polynomial $g'(x)$ which is a scalar multiple of $g(x)$. Therefore $g'(x)$ is also a right divisor of $x^n-1$ in $F_{q}[x;\theta]$. If $x^n-1=h'(x)g'(x)$
		in $F_{q}[x;\theta]$, then as a result of Theorem \ref{teo4} and Theorem \ref{teo5}; $h'^R(x)$ is a $\theta$-palindromic polynomial. Hence $C^\perp= \langle h'^R(x) \rangle$ is a reversible DNA code, by 		Theorem \ref{teo1}.
	\end{enumerate}\end{proof}
\begin{example} Let $\alpha$ be the primitive element of $F_{16}$ where $\alpha^4=\alpha+1$. Then, $x^6-1=h(x)g(x)=(1+\alpha^7x+\alpha^{7}x^2+x^3)(1+\alpha^7x+\alpha^{13}x^2+x^3)$ in $F_{16}[x;\theta]$.
Since the degree of $g(x)$ is odd and it is a $\theta$-palindromic polynomial thus the skew cyclic code $C=\langle g(x)\rangle$ is a reversible DNA code over $F_{16}$ with the parameters $[6,3,4]$ (which is an optimal code with respect to the Singleton bound).

$C^\perp$ is also a skew cyclic code with the generating polynomial $h^R(x)=1+\alpha^{13}x+\alpha^{7}x^2+x^3$ which is a $\theta$-palindromic polynomial,  so $C^\perp$ is a reversible DNA code.
\end{example}

\begin{example}Let $\alpha$ be the primitive element of $F_{16}$ where $\alpha^4=\alpha+1$. Then,
$$x^{10}-1=h(x)g(x)=(1+\alpha x+\alpha^3 x^2+\alpha x^3+x^4)(1+\alpha x+\alpha^{11} x^2+\alpha^{11} x^4+\alpha x^5+x^6)$$ in $F_{16}[x;\theta]$.
Since the degree of $g(x)$ is even and it is a palindromic polynomial, then the skew cyclic code $C=\langle g(x)\rangle$ is a reversible DNA code of length $10$ over $F_{16}$.

The dual code of $C$ is the skew cyclic code $C^\perp=\langle h^R(x)\rangle =\langle (1+\alpha^{4}x+\alpha^{3}x^2+\alpha^{4}x^3+x^4) \rangle$. Since $h^R(x)$ is a palindromic polynomial of even degree, then $C^\perp$ is also a reversible DNA code.
\end{example}
\subsection{Odd length case}


In \cite{fq+vfq}, factors of $x^n-1$ in $F_q[x;\theta]$ are determined for the case $(n,m)=1$, where $m$ is the order of $\theta$. In our case $q=4^{2s}$ and the order of $\theta$ is $2$, so the following lemma is a direct consequence of  Lemma 2 in \cite{fq+vfq}.
\begin{lemma}
  Let $g(x)$ be a right divisor of $x^n-1$ in $F_{4^{2s}}[x;\theta]$, where the order of $\theta$ is $2$ and $n$ is an odd number. Then, $g(x)$ is a polynomial over $F_{4^s}$. Moreover the factorization of $x^n-1$ in $F_{4^{2s}}[x;\theta]$ is same as the factorization of $x^n-1$ in the commutative ring $F_{4^s}[x]$.
\end{lemma}
As a result of this lemma; if $C$ is a skew cyclic code of odd length over $F_{4^{2s}}$ with respect to the automorphism $\theta$, then $C$ is an ordinary cyclic code over $F_{4^{2s}}$ generated by a polynomial with coefficients from $F_{4^s}$.
$F_{4^s}$ is the fixed subfield of $F_{4^{2s}}$ under $\theta$, since $\theta(\beta)=\beta^{4^s}=\beta$ for all $\beta\in F_{4^s}$. If $g(x)\in F_{4^{2s}}[x;\theta]$ is a palindromic polynomial with coefficients from $F_{4^s}$, then it is also a $\theta$-palindromic polynomial.

Following theorems can be easily proven by using  similar arguments  given in Subsection \ref{32}

\begin{theorem}
  Let $x^n-1=h(x)g(x)$ in $F_{4^{2s}}[x;\theta]$ where $n$ is odd. Then, the skew cyclic code  $C=\langle g(x) \rangle$ of length $n$  over $F_{4^{2s}}$ is a reversible DNA code if and only if $g(x)$ is a palindromic polynomial.
\end{theorem}
\begin{theorem}
Let $C$ be a skew cyclic code generated by a divisor  of $x^n-1$ in $F_{4^{2s}}[x;\theta]$ where n is odd. If $C$ is a reversible DNA code, then its dual code $C^\perp$ is also a reversible DNA code. 
\end{theorem}
\begin{example}
Let $\alpha$ be a primitive element of $F_{16}$. Then, the fixed subfield under $\theta$ is $F_4=\{0,1,\alpha^5,\alpha^{10}\}$.
  $$x^5-1=(x-1)(x^2+\alpha^5x+1)(x^2+\alpha^{10}x+1)\text{in } F_{16}[x;\theta]$$
  Let $ g(x)=x^2+\alpha^{10}x+1$ then the skew cyclic code $C=\langle g(x)\rangle$ is both a reversible DNA code and a reversible code over $F_{16}$ with the parameters [5,3,3]. Here, $h(x)=h^R(x)=x^3+\alpha^{10}x^2+\alpha^{10}x+1$. Since $h^R(x)$ is a palindromic polynomial, then the dual code  $C^\perp=\langle h^R(x) \rangle$ is a reversible DNA code.
\end{example}
\section{Conclusion}\label{sec_conc}
For the first time, to the best knowledge of the authors, DNA codes over non commutative rings are explored here. The non commutativity  property gives a very suitable presentation for such codes. Different from the previous studies (\cite{ise,Gse}), here we introduce a more algebraic approach provided by the skewness of the ring in studying DNA codes. Further studies as in the case for commutative rings are still open and interesting problems. For instance; GC content of DNA codes, DNA codes with respect to the other distances (edit distance, similarity distance etc.) awaits to be explored.

\textbf{Acknowledgement:} The authors wish to thank the anonymous reviewers for their valuable remarks that led to an improved presentation of our original paper.


\begin{thebibliography}{99}

\bibitem{tehergf4} T. Abulraub, A. Ghrayeb and X. Nian Zeng, \textit{Construction of cyclic codes over $GF(4)$ for
	DNA computing}, J. Franklin Inst.  Vol. 343,
No. 4-5, pp. 448--457,  2006.
\bibitem{adleman94} L. Adleman, \textit{ Molecular computation of
solutions to combinatorial problems}, Science, New Series, Vol. 266, pp. 1021--1024, 1994.
\bibitem{agul} A. Bayram, E.S. Oztas, I. Siap, Codes over $F_4 + v F_4$ and some DNA applications,
Designs,Codes and Cryptography, DOI 10.1007/s10623-015-0100-8, (2015).

\bibitem {note on dual} D. Boucher and F.Ulmer, \textit{A note on the dual codes of module skew codes}, Vol. 7089, pp. 230--243, 2011.
\bibitem {skew cyclic} D. Boucher, W. Geiselmann and F. Ulmer, \textit{Skew cyclic codes}, Appl. Algebra Eng. Comm., Vol. 18, No. 4, pp. 379--389, 2007.


\bibitem {coding with} D. Boucher and F. Ulmer, \textit{Coding with skew polynomial rings}, J. Symb. Comput., Vol. 44, No. 12,
pp. 1644--1656, 2009.
\bibitem {fq+vfq} F. Gursoy, I. Siap, B. Yildiz, \textit{Construction of skew cyclic codes over  $F_q + vF_q$}, Adv. Math. Commun., Vol. 8 No. 3, pp. 313--322, 2014.


\bibitem {Mc} B. R. McDonald, \textit{Finite rings with identity}, Marcel Dekker Inc., New York, 1974.
\bibitem{ise} E.S. Oztas, I. Siap, \textit{Lifted polynomials over $F_{16}$ and their applications to DNA Codes}, Filomat, Vol. 27, No. 3, pp. 459--466, 2013.
\bibitem{Gse} E.S. Oztas, I. Siap, \textit{On a generalization of lifted polynomials over finite fields and their applications to DNA codes}, Int. J. Comput. Math. Vol. 92, No. 9, pp. 1976--1988, 2015.
\bibitem {I.siap}I. Siap, T. Abualrub, N. Aydin and P. Seneviratne, \textit{Skew cyclic codes of arbitrary length}, Int. J.Inform. Coding Theory, Vol. 2, No. 1, pp. 10--20, 2011.
\bibitem{tahersiapsim} I. Siap, T. Abulraub  and  A. Ghrayeb, \textit{Similarity cyclic DNA codes over
rings}, International Conference on Bioinformatics and Biomedical
Engineering, in Shanghai, PRC, May 16-18th 2008.

\bibitem{siaptaherr} I. Siap, T. Abulraub and A. Ghrayeb, \textit{Cyclic DNA codes over the ring $F_2[u]/(u^2-1)$ based on the deletion distance},  J. Franklin Inst., Vol. 346, No. 8, pp. 731--740, 2009.
\bibitem{yildizsiap} B. Yildiz, I. Siap, \textit{Cyclic codes over $F_2[u]/(u^4-1)$ and applications to DNA codes}, Comput. Math. Appl., Vol. 63, No. 7,  pp. 1169--1176, 2012.
\end{thebibliography}
\end{document}